\documentclass[11pt]{amsart}

\usepackage {enumerate}
\usepackage{setspace}
\usepackage{caption}
\usepackage{mathrsfs}
\usepackage{hyperref}

\usepackage{color}

\newtheorem{theorem}{Theorem}[section]
\newtheorem{lemma}[theorem]{Lemma}
\newtheorem{corollary}[theorem]{Corollary}
\newtheorem{proposition}[theorem]{Proposition}

\theoremstyle {definition}

\newtheorem{remark}[theorem]{Remark}

\usepackage[margin=1.1in]{geometry}

\allowdisplaybreaks
  
\DeclareMathOperator{\area}{area} 

\newcommand{\R}{\mathbb{R}}
\DeclareMathOperator{\Ric}{Ric}

\DeclareMathOperator{\barea}{area_{\bar g}}

\title[Uniqueness of stable CMC surfaces] 
{Global uniqueness of large stable CMC spheres in asymptotically flat Riemannian three-manifolds} 
 
\author{Otis Chodosh}
\address{Stanford University, Department of Mathematics, Building 380, Stanford, CA 94305, United States}
\email{ochodosh@stanford.edu}

\author{Michael Eichmair}
\address{University of Vienna, Faculty of Mathematics, Oskar-Morgenstern-Platz 1, 1090 Vienna, Austria}
\email {michael.eichmair@univie.ac.at}

\begin{document}

\begin{abstract} Let $(M, g)$ be a complete Riemannian $3$-manifold that is asymptotic to Schwarzschild with positive mass and whose scalar curvature vanishes. We \textsl{unconditionally} characterize the large, embedded stable constant mean curvature spheres in $(M, g)$. 
\end{abstract}

\maketitle

\date{}


\section {Introduction} \label{sec:introduction}

The purpose of this paper is to complete the classification of large, embedded stable constant mean curvature spheres in initial data sets of general relativity. We briefly review some background before stating our results.

Let $(M, g)$ be a connected, complete Riemannian $3$-manifold. 

We usually require that $(M, g)$ is $C^{k}$-\textsl{asymptotic to Schwarzschild} with mass $m>0$. This means that there is a non-empty compact set $K \subset M$ and a diffeomorphism 
\begin{align*}
M \setminus K \cong \{x \in \R^3 : |x| > 1/2\}
\end{align*}
such that, in this \textsl{chart at infinity}, there holds
\begin{align} \label{eqn:stronglyAF}
g_{ij} = \Bigl(1 + \frac{m}{2 |x|} \Bigl)^4 \delta_{ij} + \sigma_{ij}
\end{align}
where 
\[
\partial_I \sigma_{ij} = O (|x|^{-2 - |I|}) \qquad \text{ as} \qquad |x| \to \infty
\]
for every multi-index $I$ of order $|I| \leq k$. Moreover, we usually require that $(M, g)$ has \textsl{horizon boundary}. This means that $\partial M$ is a possibly empty minimal surface and that every closed minimal surface in $(M, g)$ is contained in $\partial M$. Together, these assumptions imply that $M$ is diffeomorphic to $\R^3$ with finitely many balls removed; see e.g.~Section 4 in \cite{Huisken-Ilmanen:2001}.

Let $\Sigma \subset M$ be a closed two-sided surface. The \textsl{Hawking mass} of such a surface is the quantity 
\[
m_H (\Sigma) = \sqrt { \frac{ \area (\Sigma)}{ 16 \pi} } \left( 1 - \frac{1}{16 \pi}\int_\Sigma H^2 \, d \mu \right)
\]
where $H$ is the mean curvature scalar of $\Sigma$ associated with a designated unit normal $\nu$. In a context where this makes sense, we choose $\nu$ to be the \textsl{outward pointing} unit normal. Our conventions are such that a unit sphere in $\R^3$ has mean curvature $2$.
We also recall  that the mean curvature $H$ is constant if and only if $\Sigma$ is a critical point for the area functional among volume preserving deformations. Moreover, such a constant mean curvature surface is a \textsl{stable} critical point for this variational problem if and only if 
\[
\int_\Sigma \left(  |h|^2 + \Ric(\nu, \nu) \right) u^2 d \mu \leq \int_\Sigma |\nabla u|^2 d \mu \quad \text{ for all $u \in C^\infty(\Sigma)$ with $\int_\Sigma u \, d \mu = 0$}
\]
where $h$ is the second fundamental form of $\Sigma$ and $\Ric$ is the Ricci tensor of $(M, g)$. In \cite{Christodoulou-Yau:1988}, D.~Christodoulou and S.-T.~Yau have observed that $m_H (\Sigma) \geq 0$ for every stable constant mean curvature sphere $\Sigma \subset M$ provided that the scalar curvature $R$ of $(M, g)$ is non-negative; cf.~Theorem \ref{theo:CY-sphere-intro} below. This insight has initiated the use of such spheres and their Hawking mass to test the strength of the gravitational field in the case where $(M, g)$ arises as a maximal Cauchy surface of a spacetime.

Assume now that $(M, g)$ is asymptotic to Schwarzschild with mass $m > 0$. A foundational result of G.~Huisken and S.-T.~Yau \cite{Huisken-Yau:1996} shows that the complement of a compact subset of $M$ admits a foliation by distinguished stable constant mean curvature spheres. This \textsl{canonical foliation} gives rise to a definition of a \textsl{geometric center of mass} of $(M, g)$. The characterization of these special spheres has been refined further in important work by J.~Qing and G.~Tian \cite{Qing-Tian:2007}. We summarize these results in Appendix \ref{sec:cond-results}. These characterizations have been extended in works by J.~Metzger and the second-named author \cite{stablePMT}, S.~Brendle and the second-named author \cite{Brendle-Eichmair:2014}, A.~Carlotto and the authors \cite{mineffectivePMT}, in our paper \cite{CE:far-outlying}, and the recent paper by \cite{EichmairKoerber} by T.~Koerber and second-named author. 

Here we complete this line of inquiry by establishing the following \textsl{unconditional uniqueness} result for the leaves of the canonical foliation.  

\begin{theorem} \label{thm:scal-zero-global}
Let $(M, g)$ be a connected, complete Riemannian $3$-manifold whose scalar curvature vanishes and which is $C^{5}$-asymptotic to Schwarzschild of mass $m>0$ with horizon boundary. Every connected, closed, embedded stable constant mean curvature surface in $(M, g)$ of large enough area is a leaf of the canonical foliation. 
\end{theorem}

The assumptions on $(M, g)$ in Theorem \ref{thm:scal-zero-global} are optimal in a number of ways: 
\begin{enumerate} [(i)]
\item The assumption \textsl{vanishing scalar curvature} cannot be relaxed to the assumption \textsl{non-negative scalar curvature}. A counterexample is given in our paper \cite{CE:far-outlying}. The subtle role of scalar curvature in this problem was elucidated in earlier work \cite{Brendle-Eichmair:2014} by S.~Brendle and the second-named author.
\item The assumption that $(M, g)$ be \textsl{asymptotic to Schwarzschild with mass $m >0$} cannot be relaxed to \textsl{asymptotically flat}. Indeed, A.~Carlotto and R.~Schoen \cite{Carlotto-Schoen} have constructed non-flat examples of asymptotically flat metrics on $\R^3$ with vanishing scalar curvature that are equal to the Euclidean metric in a half-space. Note that all the coordinate spheres in the Euclidean half-space are stable constant mean curvature spheres. In particular, the conclusion of Theorem \ref{thm:scal-zero-global} fails dramatically in these examples. However, these asymptotically flat manifolds \textsl{do admit} a canonical foliation by stable constant mean curvature spheres, as has been shown by Ch.~Nerz \cite{Nerz:2014}. Even in this setting, the leaves of this foliation are the \textsl{unique} solutions of the isoperimetric problem for the volume they enclose, as has been shown in our joint paper with Y.~Shi and H.~Yu \cite{CESY}. 
\end{enumerate}

The classification of connected, closed, embedded constant mean curvature surfaces in one half of \textsl{exact} spatial Schwarzschild, 
\[
M = \{x \in \R^3 : |x| \geq m/2\} \qquad \text{ and } \qquad g_{ij} = \Big(1 + \frac{m}{2 |x|} \Big)^4 \delta_{ij},
\] 
was a long-standing problem that has been solved by S.~Brendle \cite{Brendle:2013}. The only such surfaces are the centered coordinate spheres $\{x \in \R^3 : |x| = r\}$. Brendle's method uses the \textsl{exact} warped-product structure of Schwarzschild  in an essential way. In pioneering earlier work, H.~Bray \cite{Bray:1997} has characterized these spheres as the unique solutions of the isoperimetric problem in exact Schwarzschild. J.~Metzger and the second-named author have extended Bray's approach to \textsl{asymptotically} Schwarzschild manifolds in \cite{isostructure, hdiso}.

The proof of Theorem \ref{thm:scal-zero-global} is given in Section \ref{sec:global}. The main new ingredient is Theorem \ref{thm:main}, whose proof in Section \ref{sec:proof-thm-main} occupies the bulk of this paper. To state this result, we use $S_r$ to denote the surface in $M$ that corresponds to $\{x \in \R^3 : |x| =r \}$ and by $B_r$ the bounded open region in $M$ that is enclosed by $S_r$. Given a subset $A \subset M$, we let  
\[
r_0(A) = \sup \{ r > 1 : A \cap B_r = \emptyset\}.
\]
We say that a connected, closed surface $\Sigma \subset M$ is \textsl{outlying} if it bounds a compact region in $M$ that is disjoint from $B_1$. 

\begin{theorem} \label{thm:main}
Let $(M,g)$ be a connected, complete Riemannian $3$-manifold whose scalar curvature is non-negative and which is $C^{2}$-asymptotic to Schwarzschild with mass $m>0$. There is a constant $\eta >0$ with the following property. For every connected, closed stable constant mean curvature surface $\Sigma\subset M$ that is outlying, we have that
\begin{align*} 
r_0(\Sigma) H(\Sigma) \geq \eta
\end{align*}
provided that $\area (\Sigma)$ and $r_0 (\Sigma)$ are sufficiently large. 
\end{theorem}

We briefly describe the main ideas of the proof of Theorem \ref{thm:main}.  

Consider $\Sigma \subset M$ a large stable constant mean curvature spheres, where $(M, g)$ is as in Theorem \ref{thm:main}. Assume that $\area(\Sigma)$ is large. Previously deployed strategies cannot handle the insidious combination of $\Sigma$ outlying \textsl{and} slow divergence, i.e.~$r_0(\Sigma) H(\Sigma)$ small. On the one hand, the flux integrals used in \cite{Huisken-Yau:1996, Qing-Tian:2007} all vanish simultaneously on outlying surfaces. On the other hand, the Lyapunov--Schmidt analysis applied in \cite{Brendle-Eichmair:2014} cannot handle slow divergence; the background non-linearity is too strong. 

Note that in \cite{Huisken-Yau:1996,Qing-Tian:2007,Brendle-Eichmair:2014,CE:far-outlying}, stability is only used to obtain \textsl{roundness estimates}, while centering is shown using first variation. By contrast, the centering mechanism we discover and put to good use in the proof of Theorem \ref{thm:main} is based on stability. It leverages a surprising tension between the \emph{sharp} classical Minkowski inequality and the following estimate for the Hawking mass derived from the stability of $\Sigma$.

\begin{theorem} [D.~Christodoulou and S.-T.~Yau \cite{Christodoulou-Yau:1988}] \label{theo:CY-sphere-intro}
Let $(M, g)$ be a Riemannian $3$-manifold and $\Sigma \subset M$ be a stable constant mean curvature sphere. Then
\begin{equation}\label{eq:CY-intro}
\frac 23 \int_\Sigma (R + |\mathring h|^2) d\mu \leq 16\pi - \int_\Sigma H^2 d\mu.
\end{equation}
Here, $\mathring h$ is the trace-free part of the second fundamental form $h$ of $\Sigma$.
\end{theorem}

In the context of the proof of Theorem \ref{thm:main}, if we compare the quantity on the right-hand side of \eqref{eq:CY-intro} with its Euclidean counterpart 
\[
16 \pi - \int_\Sigma \bar H^2 d \bar \mu = - 2 \int_\Sigma |\mathring {\bar h}|^2 d \bar \mu,
\]
we can isolate a favorable term 
\begin{align} \label{playoffSchwarzschild}
m^2  \int_{\Sigma} \frac{\bar g(X,\bar\nu)^{2}}{|x|^{6}} d\bar\mu
\end{align}
owing to the Schwarzschild background. An upper bound for \eqref{playoffSchwarzschild} gives a lower bound on $r_0(\Sigma)$. 

In order to capitalize on the potential of \eqref{playoffSchwarzschild}, we have to keep in check the deviation from constant of the \textsl{Euclidean} mean curvature of $\Sigma$. Control in $L^{2}$ of this deviation is sufficient for us, so at first pass, we refer to a quantitative version of Schur's lemma due to C.~De Lellis and S.~M\"uller \cite{DeLellis-Muller:2005}. This leads to the a priori estimate 
\[
\bar H = H \, O (1)
\]
for the Euclidean mean curvature. In conjunction with Allard's theorem and a blow-down argument, this estimate gives us strong analytic control on $\Sigma$. More precisely, we see that $\Sigma$ is close in $C^{1, \alpha}$ to a large coordinate sphere in the chart at infinity. 

The trouble with the application of the De Lellis--M\"uller estimate above is that it uses up too much of the favorable terms   
\begin{align} \label{errorbendingaux}
\int_\Sigma |\mathring {\bar h}|^2 d \bar \mu  \qquad \text{ and } \qquad  \int_{\Sigma} |\mathring h|^{2}d\mu
\end{align}
coming from the Euclidean bending energy and the Christodoulou--Yau estimate. However, at this point, the proximity in $C^{1,\alpha}$ of $\Sigma$ to a large coordinate sphere allows us to show that the classical Minkowski inequality -- with the \textsl{sharp} constant -- \textsl{almost} holds for $\Sigma$; see Appendix \ref{sec:Minkowski}. Using this in place of the De Lellis--M\"uller estimate, we obtain improved roundness estimates for $\Sigma$. We can then absorb all remaining error terms into the favorable terms \eqref{errorbendingaux}. This leads to the desired estimate.

\subsection*{Acknowledgments} 

We are grateful to Hubert Bray, Simon Brendle, Gerhard Huisken, and Jan Metzger for many helpful conversations, as well as the referees and Thomas Koerber for their helpful suggestions concerning the exposition. Otis Chodosh has been supported at various times by the EPSRC grant EP/K00865X/1, the Oswald Veblen Fund, the NSF grants No.~1638352 and No.~1811059/2016403, a Sloan Fellowship, and a Terman Fellowship. Michael Eichmair has been supported by the START-Project Y963-N35 of the Austrian Science Fund. 


\section{Some properties of far-out stable CMC surfaces} \label{sec:far-out-CMC}

The following main result of this section extends the uniqueness theorems for the leaves of the canonical foliation stated in Appendix \ref{sec:cond-results} to surfaces of arbitrary genus.

\begin {proposition} \label{prop:largestablecmcsphere}
Let $(M, g)$ be a Riemannian $3$-manifold that is $C^2$-asymptotically flat of rate $q > 1/2$. Every connected, closed stable constant mean curvature surface $\Sigma \subset M$ with both $r_0(\Sigma) > 1$ and $\area(\Sigma) > 1$  sufficiently large has genus zero.
\end {proposition}

Note that there are no assumptions on the sign of the scalar curvature or the boundary of $M$. 

For the proof of Proposition \ref{prop:largestablecmcsphere} and other results below, we assume for convenience and without loss of generality that $M = \R^3$ and 
\begin{align} \label{def:AF}
g_{ij} = \delta_{ij} + \zeta_{ij}
\end{align}
where
\[
\partial_I \zeta_{ij} = O(|x|^{-q-|I|}) \qquad \text{ as } |x| \to \infty
\]
for every multi-index $I$ of order $|I|=0,1,2$. The general case is modeled on such \textsl{ends}. 

Consider a sequence $\{\Sigma_k\}_{k = 1}^\infty$ of connected, closed, embedded stable constant mean curvature surfaces with 
\[
r_0 (\Sigma_k) = \sup \{ r > 1 : \Sigma_k \cap B_r = \emptyset\} \to \infty \qquad \text{ and } \qquad \area (\Sigma_k) \to \infty,
\]
where we abbreviate $B_r = B_r(0)$.

Along the way of proving Proposition \ref{prop:largestablecmcsphere}, we show that the surfaces $\Sigma_k$ are close to coordinate spheres in a sense we make precise. In the case where $\Sigma_k$ are \emph{a priori} known to have genus zero, we could follow \cite[p.~1099]{Qing-Tian:2007} instead of the argument given below.

\begin{proposition} \label{prop:area-rescaling}
Assume that $r_0(\Sigma_k)H(\Sigma_k) \to 0$. Let $\rho_k = 2/H(\Sigma_k)$ be the mean curvature radius. After passing to a subsequence, the  rescaled surfaces $\rho_k^{-1} \Sigma_k$ converge to a coordinate sphere $S_1 (\xi)$ in $\R^3 \setminus \{0\}$ with $|\xi| = 1$. More precisely, given $K \subset \R^3\setminus\{0\}$ compact, there are functions $u_k : S_1 (\xi) \to \R$ with $u_k \to 0$ in $C^3$ such that $(\rho_k^{-1} \Sigma_k) \cap K$ is contained in the radial graph of $u_k$ above $S_1(\xi)$. 
\end{proposition}

\begin{proof}
First, note that 
\begin{equation}\label{eq:sec-2-CY-all-genus}
\area(\Sigma_k) H(\Sigma_k)^2 \leq \frac{64\pi}{3} + o(1)
\end{equation}
as $k \to \infty$, by Lemma \ref{lemm:CY-all-genus}.
Moreover, 
\begin{align} \label{auxproprh0}
 |x|\, |h (x)| = O (1) \qquad \textrm{and} \qquad \sup_{r> 0} \frac{\area(\Sigma_k \cap B_r)}{r^2} = O (1).
\end{align}
by Lemma \ref{lemm:sff-bds-angst-sph} and Lemma \ref{lem:qaq}. Using also Lemma \ref{lemm:h-barh-diff-std-est}, we see that these estimates also hold if we use the Euclidean metric $\bar g$ instead of $g$. By scaling, we see that they continue to hold if we replace $\Sigma_k$ by the rescaled surface $\rho_k^{-1} \Sigma_k$. These estimates imply that locally in $\R^3 \setminus \{0\}$ and on the scale of the distance to the origin, each surface is a union of an (a priori) bounded number of graphs with (a priori) bounded $C^2$-norms. They enable us to extract convergent subsequences from $\{ \rho_k^{-1} \Sigma_k\}_{k = 1}^\infty$ in the class of proper, pointed immersions into $\R^3 \setminus \{0\}$. (See e.g.~\cite{Cooper} for very general convergence results of this type.) To ensure that the limit is non-empty, we need to choose the base points with care. To that end, let 
\[
r_1(\Sigma_k) : = \inf\{r > 1 : \Sigma_k \subset B_r\}
\]
and note that $H(S_r) \sim 2/r$ for $r > 1$ large. By the maximum principle, 
$\rho_k \leq O(1)r_1 (\Sigma_k)$. After passing to subsequence, we may thus choose $x_k \in \rho_k^{-1} \Sigma_k$ such that $\{ x_k \}_{k = 1}^\infty$ converges to a point $\neq 0$. Passing to a further subsequence, we extract a limiting  connected, proper stable constant mean curvature immersion 
\[
\varphi_\infty : \Sigma_\infty \to \R^3 \setminus \{0\}
\]
with base point $x_\infty \in \Sigma_\infty$ so $\varphi_\infty (x_\infty) = \lim_{k \to \infty}  x_k$. Note that the mean curvature of this immersion is equal to $2$. The bounds \eqref{auxproprh0} descend to this immersion. We apply a variation of the result of J.~Barbosa and M.~do Carmo \cite{Barbosa-doCarmo:1984} as developed (to handle the singularity at the origin; see also \cite[Lemma 18]{CEV}) by F.~Morgan and M.~Ritor\'e \cite{Morgan-Ritore:2002} to show that $\varphi_\infty$ is totally umbilic. From the assumption that $r_0(\Sigma_k)H(\Sigma_k)\to0$, we see that 
\[
\overline{\varphi_\infty(\Sigma_\infty)} = S_1(\xi)
\]
where $|\xi| =1$. Considering different choices of base points $x_k$, we see that $\rho_k^{-1} \Sigma_k$ converges to the union of unit radius coordinate spheres, though possibly with multiplicity. However, by \eqref{eq:sec-2-CY-all-genus}, the area of the limit counted with multiplicity is at most $16\pi /3<8\pi$. It follows that there is exactly one such limiting sphere and that the convergence occurs with multiplicity one.\footnote{Alternatively, one could avoid \eqref{eq:sec-2-CY-all-genus} here and instead use that the union of two disjoint spheres in $\R^3$ is an \emph{unstable} constant mean curvature surface.} This completes the proof. \end{proof}

In the case where $ \liminf_{k \to \infty} r_0(\Sigma_k)H(\Sigma_k) > 0$, we show below that the surfaces $\Sigma_k$ can be captured by standard methods. 

\begin{proposition}\label{prop:area-rescaling-trivial}
Assume that $\liminf_{k\to\infty} r_0(\Sigma_k)H(\Sigma_k) >0$. For every sufficiently large $k$, the rescaled surface $r_0(\Sigma_k)^{-1}\Sigma_k$ is smoothly close to a coordinate sphere in $\R^3 \setminus B_1$ with bounded, possibly small radius. In particular, $\Sigma_k$ has genus $0$.
\end{proposition}
\begin{proof}
Assume first that $ r_0(\Sigma_k) H(\Sigma_k) = O (1)$. Passing to a subsequence, the mean curvature of the rescaled surfaces has a positive limit. Curvature estimates for stable constant mean curvature surfaces with bounded mean curvature (see e.g.~\cite[Proposition 2.2]{stablePMT}) together with \cite{Barbosa-doCarmo:1984} show that the rescaled surfaces converge to a coordinate sphere in $\R^3 \setminus B_1$. In the case where $r_0(\Sigma_k)H(\Sigma_k) \to \infty$, an additional rescaling argument (see the discussion after Proposition 2.2 in \cite{stablePMT}) shows that $r_0(\Sigma_k)^{-1}\Sigma_k$ is close to a small coordinate sphere tangent to $S_1 (0)$.
\end{proof}

\begin {proof}[Proof of Proposition \ref{prop:largestablecmcsphere}]
Consider a sequence $\{\Sigma_k\}_{k=1}^\infty$ of connected, closed stable constant mean curvature surfaces $\Sigma_k \subset M$ with $r_0(\Sigma_k)\to \infty$ and $\area(\Sigma_k)\to\infty$. By Proposition \ref{prop:area-rescaling-trivial}, it suffices to only consider the case where $r_0(\Sigma_k)H(\Sigma_k)\to 0$.

We can thus apply Proposition \ref{prop:area-rescaling} to the surfaces $\Sigma_k$. Because the convergence of $\rho_k^{-1} \Sigma_k$ to $S_1(\xi)$ in Proposition \ref{prop:area-rescaling} may not be smooth across the origin, we cannot yet conclude that $\Sigma_k$ has genus $0$. Instead we argue as follows. 

Define $f_k : \Sigma_k \to\R$ by 
\[
f_k(x) = \frac {1}{2} \, |x|^2.
\]
Suppose that $x_k\in\Sigma_k$ is a sequence of critical points of $f_k$ with 
\[
\det D^2_{\Sigma_k} f_k(x_k) \leq 0.
\]
We now consider a different rescaling 
\[
\check \Sigma_k = |x_k|^{-1} \Sigma_k
\]
The same argument as in Proposition \ref{prop:area-rescaling} shows that, after passing to a subsequence, we can pass the immersions $\check\Sigma_k$ with respective base points $\check x_k = x_k/|x_k|$ to a limit to obtain a limiting connected, proper stable constant mean curvature immersion
\[
\varphi_\infty : \Sigma_\infty \to \R^3\setminus\{0\},
\]
with base point $\check x_\infty$ where $\varphi_\infty (\check x_\infty) = \lim_{k \to \infty} \check x_k$. 

When $\varphi_\infty$ has non-zero mean curvature, we can argue as in Proposition \ref{prop:area-rescaling} to show that it is a round sphere. We may then assume that this  sphere passes through the origin (otherwise, the $\check\Sigma_k$ would be smoothly converging to a sphere). When the mean curvature of $\varphi_\infty$ vanishes, we can argue as in \cite[Lemma 3.2]{stablePMT} (along with a log-cutoff argument near the origin) to conclude that it is a flat plane. Either way, it follows that $\check x_\infty$ is a critical point of $f_\infty :\Sigma_\infty \to\R$ where $f_\infty(x) = \tfrac 12 \, |\varphi_\infty(x)|^2$. From this, we see that the immersion $\varphi_\infty$ is tangent to the sphere $S_1(0)$ at the point $\tilde x_\infty$. 

Putting these facts together, we thus see that $\det D^2_{\Sigma_\infty}f_\infty(x_\infty) > 0$. Indeed, the planar case will be a strict local minimum while the spherical case will be a strict local maximum thanks to the observation that the sphere passes through the origin. This contradicts the assumption that $\det D^2_{\Sigma_k}f_k(x_k) \leq 0$. It follows that for sufficiently large $k$, $f_k$ is a Morse function on $\Sigma_k$ with no saddle points. Standard Morse theory shows that $\Sigma_k$ is a sphere. 
\end {proof}


\section{Proof of Theorem \ref{thm:scal-zero-global}} \label{sec:global}

Let $(M, g)$ be a Riemannian $3$-manifold as in the statement of Theorem \ref{thm:scal-zero-global}. Let $\{\Sigma_k\}_{k = 1}^\infty$ be a sequence of connected, closed, embedded stable constant mean curvature surfaces in $(M, g)$ with $\area(\Sigma_k) \to \infty$. Our goal is to prove that $\Sigma_k$ is a leaf of the canonical foliation of $(M, g)$ provided that $k$ is large enough.  

Note that $r_0 (\Sigma_k) \to \infty$ by Theorem 1.10 in \cite{mineffectivePMT} of A.~Carlotto and the authors. We may assume that each $\Sigma_k$ is a sphere by Proposition \ref{prop:largestablecmcsphere}. If $\Sigma_k$ encloses $B_1$ and $k$ is sufficiently large, then it is a leaf of the canonical foliation. This follows from the works of G.~Huisken and S.-T.~Yau  and of J.~Qing and G.~Tian stated in Appendix \ref{sec:cond-results}. We may thus assume that the surfaces in the sequence are outlying spheres. 

Theorem \ref{thm:main} implies that 
\[
\liminf_{k\to\infty} r_0(\Sigma_k) H(\Sigma_k) > 0.
\]
Thus, by Proposition \ref{prop:area-rescaling-trivial} and provided $k$ is sufficiently large, $\Sigma_k$ is smoothly close to a large coordinate sphere, separated from the origin, in the asymptotically flat end. In particular, the spheres $\Sigma_k$ are captured by the Lyapunov--Schmidt analysis developed by S.~Brendle and the authors \cite{Brendle-Eichmair:2014, CE:far-outlying} as surveyed in the introduction to our companion article \cite{CE:far-outlying}. The questions of whether such spheres exist in $(M, g)$ or indeed can be ruled out to exist are reduced by this analysis to the study of fine properties of the scalar curvature of $(M, g)$ in the asymptotically flat end. In particular, since we have assumed that the scalar curvature vanishes outside of a compact set, no such spheres exist; cf.~Section 1 in \cite{CE:far-outlying}. Note that decay of the metric in $C^6$ is assumed in \cite{CE:far-outlying}. This has been weakened to decay in $C^5$ in \cite{EichmairKoerber}. This completes the proof of Theorem \ref{thm:scal-zero-global}. 


\section {Proof of Theorem \ref{thm:main}}\label{sec:proof-thm-main}

Let $(M, g)$ be as in the statement of Theorem \ref{thm:main}. 
We denote by  
\[
\bar g   = \sum_{i=1}^3 dx^i \otimes dx^i  \qquad \text{ and } \qquad g_S = \Big( 1 + \frac{m}{2 |x|}\Big)^4 \sum_{i=1}^3 dx^i \otimes dx^i
\]
the exact Euclidean and the exact Schwarzschild metric, and by    
\[
X = x^i \partial_i
\]
the position vector field in the chart at infinity. We will compare geometric quantities with respect to $g$, the Euclidean metric $\bar g$, and the Schwarzschild metric $g_S$. To set them apart, we use a bar to denote Euclidean quantities and a subscript $S$ for Schwarzschild quantities. 

Let $\Sigma \subset M$ be a connected, closed stable constant mean curvature surface that is outlying. We identify $\Sigma$ with a surface in $M \setminus B_1 \cong \R^3 \setminus B_1(0)$ using the chart at infinity. Note that $\Sigma$ bounds a compact region in $\R^3$ that is disjoint from $B_1(0)$. We use $d \mu$, $h$, and $H$ to denote the area measure, the second fundamental form, and the mean curvature (with respect to the outward pointing unit normal) of $\Sigma$. 

Assume, for a contradiction, that there is a sequence $\{\Sigma_k\}_{k=1}^\infty$ of such surfaces in $(M, g)$ with  
\[
\area(\Sigma_k)\to\infty, \qquad \qquad r_0(\Sigma_k)\to \infty, \qquad \text{ and } \qquad r_0(\Sigma_k)H(\Sigma_k)\to0.
\]

By Proposition \ref{prop:largestablecmcsphere}, we may assume that each $\Sigma_k$ is a sphere. We collect estimates for these surfaces from Appendix \ref{app:curvature-area-growth-CMC}. All the error terms $o (1)$ and $O (1)$ below are with respect to $k \to \infty$. 

For every $\gamma > 2$,
\begin{align} \label{errordecayHY}
r_{0}(\Sigma_k)^{- 2 + \gamma } \int_{\Sigma_k} |x|^{-\gamma} d \mu= O (1)
\end{align}
by Lemma \ref{lem:growth-conditions}. We have that
\begin{equation}\label{eqn:curvaturestimate}
|x| \, |h(x)| = O (1)
\end{equation}
by Lemma \ref{lemm:sff-bds-angst-sph}. By Lemma \ref{lem:qaq}, 
\begin{equation}\label{eq:qag}
\sup_{r > 1} \frac{\area(\Sigma_k\cap B_r)}{r^2} = O (1).
\end{equation}
Trivially,
\begin{equation}\label{eq:eqaq}
\frac 12 \barea(\Sigma_k\cap B_r) \leq \area(\Sigma_k\cap B_r) \leq 2 \barea(\Sigma_k\cap B_r).
\end{equation}
Finally, by Lemma \ref{lemm:HYH2},  
\begin{align}
\area (\Sigma_k) H(\Sigma_k)^2  &= 16\pi + o (1) \label{eq:HYH} \\
\int_{\Sigma_k} |\mathring h|^2 d\mu &= o(1) \label{eq:HY-tfsff}.
\end{align}
\begin{lemma} [Area element comparison]
We have that
\[
d\mu = \Big( 1 + \frac{m}{2 |x|}\Big)^4 \Big( 1 + O(|x|^{-2})\Big) d\bar\mu.
\]
\end{lemma}


\begin{lemma} [Mean curvature comparison]
We have that 
\begin{align} \label{comparisonmc-coarse}
H &= \bar H + O(|x|^{-2}) 
\end{align}
and 
\begin {align}
\Big( 1 + \frac{m}{2|x|}\Big)^{2} H &= \bar H - \Big( 1 + \frac{m}{2|x|}\Big)^{-1} \frac{2m}{|x|^{3}} \, \bar g(X,\bar\nu) + O(|x|^{-3}). \label{comparisonmc}
\end {align}
\begin{proof}
A standard computation as in \cite[p.~418]{Huisken-Ilmanen:2001} gives that 
\[
H = \bar H  + O(|\partial \sigma|) + O(|\bar h| |\sigma|).
\]
The first estimate now follows from \eqref{eqn:curvaturestimate} and Lemma \ref{lemm:h-barh-diff-std-est}. For the second estimate, consider the metric
\[
\hat g = \bar g  + \hat \sigma \qquad \text{ where } \qquad  \sigma =   \Big( 1 + \frac{m}{2|x|}\Big)^{4} \hat \sigma
\]
that is conformally related to $g$. Note that 
\[
\Big( 1 + \frac{m}{2|x|}\Big)^{2} H =  \hat H - \Big( 1 + \frac{m}{2|x|}\Big)^{-1} \frac{2m}{|x|^{3}} \, \hat g (X, \hat \nu)
\]
by the formula (see e.g.~\cite [Lemma 1.4]{Huisken-Yau:1996}) for the change of mean curvature under conformal changes of the metric. The same computation as above gives
\[
\hat H = \bar H  + O(|\partial \hat \sigma|) + O(|\bar h| |\hat \sigma| ).
\]
The asserted estimate now follows in conjunction with \eqref{eqn:curvaturestimate} and Lemma \ref{lemm:h-barh-diff-std-est}.
\end {proof}
\end{lemma}


\begin{lemma}
Let $\delta \in (0,1)$. We have that 
\begin{align} \label{boundsffaux}
(1-\delta) \int_{\Sigma_k} |\mathring{\bar h}|^{2}_{\bar g} d\bar \mu \leq \int_{\Sigma_k} |\mathring h|^{2}d\mu  + O(r_{0}(\Sigma_k)^{-4}).
\end{align}
\end{lemma}
\begin{proof}
First, note that 
\[
\int_{\Sigma_k} |\mathring {\bar h}|^{2}_{\bar g}d\bar \mu = \int_{\Sigma_k} |\mathring { h}_S|_{S}^2 d\mu_S
\]
by conformal invariance. Moreover, 
\[
\int_{\Sigma_k} |\mathring h|^{2}d\mu_{S} = (1+o(1))\int_{\Sigma_k} |\mathring h|^{2}d\mu.
\]
The curvature bound \eqref{eqn:curvaturestimate} and a standard computation as in \cite[p.~418]{Huisken-Ilmanen:2001} gives
\[
|\mathring h |  = |\mathring h_{S} |_{S}  + O(|x|^{-3}).
\]
Using Young's inequality, we obtain
\[
 |\mathring h_{S} |_{S}^2  \leq |\mathring h |^2 + \frac  \delta 4  |\mathring h_{S} |_{S}^{2} + O(|x|^{-6})
\]
where $\delta^{-1}$ has been absorbed into the error term. Using these estimates together with \eqref{errordecayHY}, we obtain
\[
\int_{\Sigma_k} |\mathring {\bar h}|^{2}_{\bar g}d\bar \mu =  \int_{\Sigma_k} |\mathring h_{S}|_{S}^2 d \mu_S  \leq  (1 + o(1)) \int_{\Sigma_k} |\mathring h|^2 d \mu  + \frac  \delta 2 \int_{\Sigma_k} |\mathring{\bar h}|_{\bar g}^{2}d\bar\mu + O(r_{0}(\Sigma_k)^{-4}).
\]
This proves the assertion.
\end{proof}

By a result of C.~De Lellis and S.~M\"uller \cite{DeLellis-Muller:2005}, there is a constant $\Gamma > 1$ such that 
\begin{equation}\label{eq:DM-arb-const}
\int_{\Sigma_k} |\bar H-2/\lambda_k|^{2} d\bar\mu \leq 2 \, \Gamma \int_{\Sigma_k} |\mathring{\bar{h}}|^{2}_{\bar g}d\bar\mu
\end{equation}
holds for appropriate choice of $\lambda_k > 0$.

\begin{lemma}
We have
\begin{align*} 
\lambda_k \, H(\Sigma_k) = 2 + o (1).
\end{align*}
\begin{proof}
Consider the rescaled surfaces $\rho_k^{-1} \Sigma_k$ as in Proposition \ref{prop:area-rescaling}. After passing to a subsequence, these surfaces converge to a coordinate sphere $S_1(\xi)$ with $|\xi| = 1$, locally smoothly in $\R^3 \setminus \{0\}$. In particular, their Euclidean mean curvature converges to $2$ away from the origin. Now, the right-hand side of \eqref{eq:DM-arb-const} is $o(1)$, both sides of \eqref{eq:DM-arb-const} are scaling invariant, and the integrand on the left is non-negative. The asserted estimate follows from these facts.
\end{proof}
\end{lemma}

We use \eqref{eq:DM-arb-const} to derive an estimate \eqref{finalHawking} for the Hawking mass of $\Sigma$ that is positioned against the Christodoulou--Yau estimate in Theorem \ref{theo:CY-sphere-intro}.\\

We have
\begin{equation*} 
 \Big( 1 + \frac{m}{2|x|}\Big)^{4} H^{2} = \Big( \bar H - \Big( 1 + \frac{m}{2|x|}\Big)^{-1} \frac{2m}{|x|^{3}} \, \bar g(X,\bar\nu) \Big)^{2} + O(H |x|^{-3}) + O(|x|^{-6})
\end{equation*}
by \eqref{comparisonmc}. Hence,
\begin{align*}
\int_{\Sigma_k} H^{2} d\mu & = \int_{\Sigma_k} \Big( \bar H - \Big( 1 + \frac{m}{2|x|}\Big)^{-1} \frac{2m}{|x|^{3}} \,  \bar g(X,\bar\nu) \Big)^{2}  \Big( 1 + \frac{m}{2|x|}\Big)^{-4} d\mu\\
& \qquad +  O \int_{\Sigma_k} (H|x|^{-3} + |x|^{-6}) d\mu\\
& = \int_{\Sigma_k} \Big( \bar H - \Big( 1 + \frac{m}{2|x|}\Big)^{-1} \frac{2m}{|x|^{3}} \,  \bar g(X,\bar\nu) \Big)^{2} \Big( 1 + O(|x|^{-2})\Big) d\bar\mu\\
& \qquad + O \int_{\Sigma_k} (H|x|^{-3} + |x|^{-6})d\mu\\
& =  \int_{\Sigma_k} \Big( \bar H - \Big( 1 + \frac{m}{2|x|}\Big)^{-1} \frac{2m}{|x|^{3}} \,  \bar g(X,\bar\nu) \Big)^{2} d\bar\mu\\
& \qquad + O \int_{\Sigma_k} (H^2 |x|^{-2} + H|x|^{-3} + |x|^{-6})d\mu.
\end{align*}
Note that we have \textsl{precisely} isolated the Schwarzschild contribution. \\

We now expand the Schwarzschild term. 
\begin{align*}
 \int_{\Sigma_k} \Big( \bar H - \Big( 1 + \frac{m}{2|x|}\Big)^{-1}\frac{2m}{|x|^{3}} \, \bar g(X,\bar\nu)\Big)^{2}d\bar\mu & = \int_{\Sigma_k} \bar H^{2}d \bar\mu \\
& \qquad-  \int_{\Sigma_k} \bar H \Big( 1 + \frac{m}{2|x|}\Big)^{-1} \frac{4m}{|x|^{3}} \, \bar g(X,\bar\nu)d\bar\mu\\
& \qquad+ \int_{\Sigma_k} \Big(1+\frac{m}{2|x|}\Big)^{-2} \frac{4m^{2}}{|x|^{6}} \, \bar g(X,\bar\nu)^{2}d\bar\mu\\
& = 16\pi +  2 \int_{\Sigma_k} |\mathring{\bar h}|^{2}_{\bar g}d\bar\mu\\
& \qquad -  \int_{\Sigma_k} \bar H \frac{4m}{|x|^{3}} \, \bar g(X,\bar\nu)d\bar\mu\\
& \qquad+  \int_{\Sigma_k} \bar H \Big( 1 - \Big( 1 + \frac{m}{2|x|}\Big)^{-1} \Big) \frac{4m}{|x|^{3}}\, \bar g(X,\bar\nu)d\bar\mu\\
& \qquad+ \int_{\Sigma_k} \Big(1+\frac{m}{2|x|}\Big)^{-2} \frac{4m^{2}}{|x|^{6}} \,  \bar g(X,\bar\nu)^{2}d\bar\mu
\intertext{Observe that  
\begin{equation} \label{eq:outlying-flux-vanishes}
\int_{\Sigma_k} \frac { \bar g (X, \bar \nu)}{|x|^3} d \bar \mu = 0
\end{equation}
since $\Sigma$ \textsl{does not} enclose the origin. We choose $\tau$ with $2 < \tau < 8/3$ and continue to estimate.}
& = 16\pi + 2 \int_{\Sigma_k} |\mathring{\bar h}|^{2}_{\bar g}d\bar\mu\\
& \qquad-  \int_{\Sigma_k} \Big(\bar H -\frac{2}{\lambda_k}\Big) \frac{4m}{|x|^{3}} \, \bar g(X,\bar\nu)d\bar\mu\\
& \qquad+  \int_{\Sigma_k} \bar H  \frac{4m^{2}}{|x|^{3}(2|x|+m)} \,  \bar g(X,\bar\nu)d\bar\mu\\
& \qquad+ \int_{\Sigma_k} \Big(1+\frac{m}{2|x|}\Big)^{-2} \frac{4m^{2}}{|x|^{6}} \,  \bar g(X,\bar\nu)^{2}d\bar\mu\\
& \geq  16\pi + 2 \int_{\Sigma_k} |\mathring{\bar h}|^{2}_{\bar g}d\bar\mu\\
& \qquad- \frac{\tau}{2}  \int_{\Sigma_k} \big|\bar H -\frac{2}{\lambda_k}\big|^{2} d\bar\mu\\
& \qquad-  \frac{2}{\tau} \int_{\Sigma_k} \frac{4m^{2}}{|x|^{6}} \,  \bar g(X,\bar\nu)^{2}d\bar\mu\\
& \qquad+  \int_{\Sigma_k} \bar H  \frac{4m^{2}}{|x|^{3}(2|x|+m)} \, \bar g(X,\bar\nu)d\bar\mu\\
& \qquad+ \int_{\Sigma_k} \Big(1+\frac{m}{2|x|}\Big)^{-2} \frac{4m^{2}}{|x|^{6}} \,  \bar g(X,\bar\nu)^{2}d\bar\mu.
\intertext{We use the estimate of De Lellis--M\"uller \eqref{eq:DM-arb-const} on the second line. We also combine the third with the last line.}
& \geq  16\pi + (2- \tau\Gamma ) \int_{\Sigma_k} |\mathring{\bar h}|_{\bar g}^{2}d\bar\mu\\
&\qquad +  \int_{\Sigma_k} \bar H  \frac{4m^{2}}{|x|^{3}(2|x|+m)} \, \bar g(X,\bar\nu)d\bar\mu\\
& \qquad+  4 m^2 \Big( 1 - \frac{2}{\tau}\Big) \int_{\Sigma_k}   \frac{ \bar g(X,\bar\nu)^{2} }{|x|^{6}} d\bar\mu\\
& \qquad+ O \int_{\Sigma_k} |x|^{-5} d\bar\mu
\intertext{We estimate the second line using \eqref{comparisonmc-coarse}.}
& =  16\pi + (2- \tau \Gamma) \int_{\Sigma_k} |\mathring{\bar h}|_{\bar g}^{2}d\bar\mu\\
& \qquad+  4 m^2 \Big( 1 - \frac{2}{\tau}\Big) \int_{\Sigma_k}   \frac{ \bar g(X,\bar\nu)^{2} }{|x|^{6}} d\bar\mu\\
& \qquad + O \int_{\Sigma_k} (H|x|^{-3} + |x|^{-5})d\bar\mu.
\end{align*}

In conclusion, we obtain 
\begin{align} \label{finalHawking}
\int_{\Sigma_k} H^{2}d\mu & \geq  16\pi + (2-  \tau\Gamma) \int_{\Sigma_k} |\mathring{\bar h}|_{\bar g}^{2}d\bar\mu\\ \nonumber
&  \qquad + 4 m^2 \Big(1 - \frac2\tau \Big) \int_{\Sigma_k} \frac{\bar g(X,\bar\nu)^{2}}{|x|^{6}} d\bar\mu\\ \nonumber
& \qquad + O\int_{\Sigma_k} (H^2 |x|^{-2} + H|x|^{-3} + |x|^{-5}) d\mu.
\end{align}

Combining Theorem \ref{theo:CY-sphere-intro} with \eqref{boundsffaux} and \eqref{finalHawking} and then estimating the resulting error terms using \eqref{errordecayHY} and \eqref{eqn:curvaturestimate},  we obtain
\begin{align}
& \left( \frac {2}{3} (1-\delta) + 2 - \tau \Gamma \right)      \int_{\Sigma_k} |\mathring{\bar h}|^{2}_{\bar g} d\bar \mu +  4 \, m^2 \Big(1 - \frac2\tau \Big)  \int_{\Sigma_k} \frac{\bar g(X,\bar\nu)^{2}}{|x|^{6}} d\bar\mu \label{eq:BIG-inequality-with-Gamma}\\ 
& \qquad  \leq O(r_{0}(\Sigma_k)^{-4})  
  + O  \int_{\Sigma_k} (H^2 |x|^{-2} + H|x|^{-3} + |x|^{-5}) d\mu \nonumber\\
& \qquad \leq O(r_{0}(\Sigma_k)^{-3}) + H(\Sigma_k)\,  O(r_{0}(\Sigma_k)^{-1}). \nonumber
\end{align}

We emphasize that the second term in \eqref{eq:BIG-inequality-with-Gamma} is owed to the Schwarzschild background. We now study this term more closely.

\begin{lemma} \label{lem:r-2-lemm}
We have that 
\begin{align*}
\liminf_{k \to \infty} \left( r_{0}(\Sigma_k)^{2} \int_{\Sigma_k} \frac { \bar g(X,\bar\nu)^{2}}{|x|^6}d\bar\mu \right)> 0.
\end{align*}
\begin{proof} 
The rescaled surfaces $r_0(\Sigma_k)^{-1}\Sigma_k$ are contained in $\{x \in \R^3 : |x| \geq 1\}$ and tangent to $S_1(0)$. By \eqref{eqn:curvaturestimate} and Lemma \ref{lemm:h-barh-diff-std-est}, \eqref{eq:qag}, and \eqref{eq:eqaq}, area and as well as extrinsic curvature of these surfaces with respect to the Euclidean background metric are locally uniformly bounded. Their mean curvature tends to zero. Using also \eqref{eq:HY-tfsff}, we see that a subsequence (with appropriate choice of basepoints) converges in the sense of pointed immersions to an affine plane $\Pi$ tangent to $S_1(0)$. Observe that
\[
c = \int_\Pi \frac{\bar g(X,\bar\nu)^2}{|x|^6} d\bar\mu > 0
\]
is independent of the particular limiting plane $\Pi$. To conclude the argument, note that the quantity in the statement of the lemma is scaling invariant and the integrand in the statement of the lemma is non-negative. This completes the proof.
\end{proof}
\end{lemma}

We can now prove a preliminary version of Theorem \ref{thm:main}. The proof uses the crucial bound  
\begin{align} \label{aux:crucialerror}
\int_{\Sigma_k} |\mathring h|^2 d \mu \leq \area (\Sigma_k)^{-1/2} \, O(1) = H(\Sigma_k)\, O(1)
\end{align}
proven in Proposition \ref{prop:IMCF-CY}.

\begin{lemma} \label{lemm:prelim-main-thm}
We have
\begin{align*}
r_{0}(\Sigma_k)^{-2} \leq  H(\Sigma_k)\, O(1).
\end{align*}

\begin{proof}
We choose $\tau = 3$ in \eqref{eq:BIG-inequality-with-Gamma}. Rearranging \eqref{eq:BIG-inequality-with-Gamma} and using Lemma \ref{lem:r-2-lemm}, we obtain\footnote{We will show below that $\Sigma_k$ satisfies the De Lellis--M\"uller estimate \eqref{eq:DM-arb-const} with constant $\Gamma = 1 + o (1)$. We may then choose $\tau$ sufficiently close to $2$ so as to arrange for the coefficient of $\int_{\Sigma_k} |\mathring{\bar{h}}|^2_{\bar g}d\bar\mu$ in \eqref{eq:BIG-inequality-with-Gamma} to be positive. At this point of the argument, however, we have to treat this term as error.} 
\[
r_{0}(\Sigma_k)^{-2} \leq O(r_{0}(\Sigma_k)^{-3}) +  O(H(\Sigma_k) r_{0}(\Sigma_k)^{-1}) + O(1) \int_{\Sigma_k} |\mathring{\bar{h}}|^{2}_{\bar g}d\bar\mu.
\]
Combined with \eqref{boundsffaux} and \eqref{aux:crucialerror}, this proves the assertion.
\end{proof}
\end{lemma}

Note that Lemma \ref{lemm:prelim-main-thm} implies that
\[
|x|^{-2} \leq O(H).
\]
Using also \eqref{comparisonmc-coarse}, we obtain the estimate 
\begin{align*} 
\bar H = O(H)
\end{align*}
for the Euclidean mean curvature. From this and \eqref{eq:HYH}, we see that the rescaled surfaces $\rho_k^{-1}  \Sigma_k$ have bounded \textsl{Euclidean} mean curvature, where $\rho_k = 2 / H (\Sigma_k)$. Together with the bounds on the area growth from \eqref{eq:qag} and \eqref{eq:eqaq}, this estimate enables us to use Allard's theorem to show that the geometric convergence proven in Proposition \ref{prop:area-rescaling} also holds across the origin. 

\begin{corollary} \label{cor:C1alpha}
Let $\alpha \in (0, 1)$. Possibly after passing to a subsequence, the rescaled surfaces $\rho_k^{-1} \Sigma_k$ converge in $C^{1,\alpha}$ to a coordinate sphere $S_{1}(\xi)$ in $\R^3$ where $\xi \in \R^3$ with $|\xi| = 1$. 
\begin {proof}
Proposition \ref{prop:area-rescaling} and the assumption that $r_0(\Sigma_k)H(\Sigma_k)\to 0$ imply subsequential convergence (smooth with multiplicity one) away from the origin. The $C^{1,\alpha}$ convergence across the origin is now a straightforward consequence of Allard's theorem (as stated in \cite[Theorem 24.2]{Simon:GMT}). 
\end {proof}
\end{corollary}

\begin{lemma} \label {lemm:sharp-DM}
We have that 
\begin{align*} 
\inf_{\lambda > 0} \int_{\Sigma_k} \big|\bar H - 2/\lambda\big|^{2} d\bar\mu \leq 2(1 + o (1)) \int_{\Sigma_k} |\mathring{\bar h}|^{2}_{\bar g}d\bar\mu. 
\end{align*}

\begin{proof}
It follows from \eqref{errordecayHY}, \eqref{eqn:curvaturestimate}, and \eqref{comparisonmc-coarse} that 
\[
\int_{\Sigma_k} \bar H^2 d \bar \mu = \int_{\Sigma_k} H^2 d \mu + O (r_0(\Sigma_k)^{-1}) = O (1).
\]
Moreover, by \eqref{eq:HY-tfsff} and \eqref{boundsffaux}, 
\[
\int_{\Sigma_k} |\mathring {\bar h}|_{\bar g}^2 d \bar \mu = o (1).
\]

Combining the Minkowski inequality proved in Proposition \ref{prop:improvedminkowsi} with Corollary \ref{cor:C1alpha}, we obtain 
\[
16 \pi \leq  (2 / \lambda_k)^2 \, \barea (\Sigma_k) + o (1) \int_{\Sigma_k} | \mathring {\bar h}|^2_{\bar g} d \bar \mu.
\]
Here, $\lambda_k >0$ satisfies
\[
\frac{2}{\lambda_k} = \frac{1}{\barea(\Sigma_k)} \int_{\Sigma_k} \bar H d \bar \mu > 0.
\]
From this, the Gauss equation, and the Gauss-Bonnet formula, we find
\begin{align*}
\int_{\Sigma_k} |\bar H-2/\lambda_k|^{2} d\bar\mu &= \int_{\Sigma_k} \bar H^{2} d\bar\mu - (2/\lambda_k)^{2}\barea(\Sigma_k) \\ &= 16 \pi + 2  \int_{\Sigma_k} | \mathring {\bar {h}}|^2_{\bar g}  d\bar\mu - (2/\lambda_k)^{2}{\barea}(\Sigma_k)  \\ &\leq 2 (1+o(1)) \int_{\Sigma_k} |\mathring{\bar h}|^{2}_{\bar g} d\bar\mu.
\end{align*}
This completes the proof. 
\end{proof}
\end{lemma}

Note that Lemma \ref{lemm:sharp-DM} just says that \eqref{eq:BIG-inequality-with-Gamma} holds with $\Gamma = 1 + o (1)$. Thus,
\begin{align*}
 & \Big( \frac {2}{3}(1-\delta) + 2 - \tau + o(1)\Big)  \int_{\Sigma_k} |\mathring{\bar h}|^{2}_{\bar g} d\bar \mu +  4 \, m^2 \Big(1 - \frac2\tau \Big)  \int_{\Sigma_k} \frac{\bar g(X,\bar\nu)^{2}}{|x|^{6}} d\bar\mu \\ 
 &\qquad \qquad\leq O(r_{0}(\Sigma_k)^{-3})  + O(H(\Sigma_k) r_0(\Sigma_k)^{-1}).
\end{align*}
We choose $\tau > 2$ and $ \delta > 0$ so that\footnote{Here it is crucial that we have kept the term on the left-hand side of the Christodoulou--Yau estimate \eqref{eq:CY-intro}!}
\[
\frac{2}{3} (1-\delta) + 2 - \tau  > 0.
\] 
It follows that
\begin{align*}
\int_{\Sigma} \frac{\bar g(X,\bar\nu)^{2}}{|x|^{6}} d\bar\mu & \leq O(r_{0}(\Sigma_k)^{-3})  + O(H(\Sigma_k)r_0(\Sigma_k)^{-1}).
\end{align*}
Together with Lemma \ref{lem:r-2-lemm}, this gives
\[
r_{0}(\Sigma_k)^{-2} \leq O(r_{0}(\Sigma_k)^{-3}) + O(H(\Sigma_k)r_{0}^{-1}).
\]
This estimate is not compatible with the assumption $r_0(\Sigma_k)H(\Sigma_k)\to 0$. This contradiction completes the proof of Theorem \ref{thm:main}. 

\appendix 


\section {The canonical foliation after Huisken--Yau and Qing--Tian} \label{sec:cond-results} 

Here we review foundational results by G.~Huisken and S.-T.~Yau \cite{Huisken-Yau:1996} and by J.~Qing and G.~Tian \cite{Qing-Tian:2007} on the canonical foliation. We apply these results in the proof of Theorem \ref{thm:scal-zero-global}.

\begin {theorem} [G.~Huisken and S.-T.~Yau \cite{Huisken-Yau:1996}] \label{Huisken-Yau:1996} 
Let $(M, g)$ be a complete Riemannian $3$-manifold that is $C^4$-asymptotic to Schwarzschild with mass $m > 0$. There is a family of distinguished embedded stable constant mean curvature spheres $\{ \Sigma_H \}_{0 < H < H_0}$ that foliate the complement of a compact subset $C \subset M$. For every $s \in (1/2, 1]$, there is $0 < H_0(s) < H_0$ with the following property. Let $0 < H < H_0(s)$. Then $\Sigma_H$ is the unique stable constant mean curvature sphere of mean curvature $H$ in $(M, g)$ that encloses the centered coordinate ball $B_{H^{-s}}$. 
\end {theorem}

In fact, the surface $\Sigma_H$ is constructed as a perturbation of the centered coordinate sphere $S_{2/H}$. We mention that R.~Ye \cite{Ye:AF-CMC} has given an alternative construction of the foliation. It has been shown by J.~Metzger \cite{Metzger:2007} that Theorem \ref{Huisken-Yau:1996} holds when $(M, g)$ is $C^2$-asymptotic to Schwarzschild with positive mass. 

We refer to $\{\Sigma_H\}_{0 < H < H_0}$ as the \textsl{canonical foliation} of the end of $(M, g)$.

The uniqueness result for the leaves of the canonical foliation has been strengthened by J.~Qing and G.~Tian \cite{Qing-Tian:2007}.

\begin {theorem} [J.~Qing and G.~Tian \cite{Qing-Tian:2007}] \label{Qing-Tian:2007}
Assumptions as in Theorem \ref{Huisken-Yau:1996}. Upon shrinking $H_0>0$ and enlarging $C$ accordingly, if necessary, the following uniqueness result holds. Let $H \in (0, H_0)$. Then $\Sigma_H$ is the unique stable constant mean curvature sphere of mean curvature $H$ that is embedded in $(M, g)$ and which encloses $C$. 
\end {theorem}

\begin {remark}
In Appendix \ref{app:QT}, we provide an alternative proof of Theorem \ref{Qing-Tian:2007} if $(M, g)$ is assumed to have non-negative scalar curvature. This proof is based on the method we develop in Section \ref{sec:proof-thm-main}.
\end {remark}

\begin {remark} 
We  prove in Proposition \ref{prop:largestablecmcsphere} that connected, closed, embedded stable constant mean curvature surfaces in $(M, g)$ are necessarily spheres, provided $r_0(\Sigma) > 1$ is sufficiently large. This extends the scope of the uniqueness statement in Theorems \ref{Huisken-Yau:1996} and \ref{Qing-Tian:2007} from spheres to closed surfaces of any genus. 
\end {remark}


\section{Curvature and area growth estimates for stable CMC surfaces} \label{app:curvature-area-growth-CMC}

Let $g$ be a metric on $\R^3$ that is $C^2$-asymptotically flat of rate $q > 1/2$ as in \eqref{def:AF}. In this section, we recall several estimates for sequences $\{\Sigma_k\}_{k = 1}^\infty$ of connected, closed, embedded stable constant mean curvature surfaces in $(\R^3, g)$ with $r_0 (\Sigma_k) \to \infty$. Some of these estimates were stated and proven in the literature under stronger asymptotic conditions -- either $q=1$ in \eqref{def:AF} or $g$ as in \eqref{eqn:stronglyAF} with $k = 2$. The proofs carry over to the present setting, sometimes with minor modifications which we indicate below. We use a bar or sometimes a subscript when quantities are computed with respect to the Euclidean background metric $\bar g$. The error terms $o(1)$ and $O (1)$ below all hold uniformly as $k \to \infty$. 

\begin{lemma}[{cf.~\cite[p.~418]{Huisken-Ilmanen:2001}}] \label{lemm:h-barh-diff-std-est}
We have that
\[
|x| \, |(h - \bar h)(x)| = O(|x|^{1-q}|h (x)|) + O(|x|^{-q}).
\]
\end{lemma}

\begin{lemma}[{\cite[Lemma 5.2]{Huisken-Yau:1996}}] \label{lem:growth-conditions}
Let $\gamma>2$. We have that  
\[
r_0(\Sigma_k)^{-2 + \gamma} \int_{\Sigma_k} |x|^{-\gamma} d\mu = O(1).
\]
\end{lemma}

\begin {lemma}[{\cite[Lemma 2.5]{stablePMT}}] \label{lemm:CY-all-genus}
We have that
\begin{align*} 
\area (\Sigma_k) H(\Sigma_k)^2  \leq \frac{64 \pi}{3} + o (1).
\end{align*} 
\end {lemma}


\begin {lemma}[{cf.~\cite[Proposition 2.3]{stablePMT}}] \label{lemm:sff-bds-angst-sph}
Assume that $r_0(\Sigma_k)H(\Sigma_k) = O (1)$. We have that 
\[
|x|\, |h(x)| = O (1).
\]
\begin {proof}
The assumption $r_0(\Sigma_k)H(\Sigma_k) = O (1)$ can be used to rule out spherical limits occurring in the ``latter alternative'' in the proof of \cite[Proposition 2.3]{stablePMT}. 
\end{proof}
\end {lemma} 


\begin{lemma}[{cf.~\cite[Corollary 2.6]{stablePMT}}]\label{lem:qaq} Assume that $r_0(\Sigma_k)H(\Sigma_k) = O(1)$. We have that 
\begin{align*}
\sup_{r > 1} \frac{\area(\Sigma_k \cap B_r)}{r^2}  = O (1).
\end{align*}
\end {lemma}
\begin {proof} 
Estimate \cite[(1.3)]{Simon:willmore} implies that
\[
\sup_{r > 1} \frac{\barea(\Sigma_k \cap B_r)}{r^2} = O (1) \,  \int_{\Sigma_k} \bar H^2 d\bar\mu.
\]
Using the previous estimates in this appendix exactly as in the proof of \cite[Lemma B.3]{stablePMT}, we obtain
\[
\int_{\Sigma_k} \bar H^2 d\bar \mu = O (1) \, \area(\Sigma_k)H(\Sigma_k)^2 = O (1). 
\]
The assertion follows since $\area(\Sigma_k\cap B_r)$ is comparable to $\barea(\Sigma_k\cap B_r)$ for large $k$. 
\end {proof}

\begin {lemma} [{cf.~\cite{Huisken-Yau:1996}}] \label{lemm:HYH2} 
Assume that the surfaces $\Sigma_k$ have genus zero. We have that  
\begin{align*}
\area (\Sigma_k) H(\Sigma_k)^2  = 16\pi + o (1) \qquad \textrm{and} \qquad \int_{\Sigma_k} |\mathring h|^2 d\mu = o(1).
\end{align*} 
\begin {proof} This is contained in the proof of Proposition 5.3 in \cite{Huisken-Yau:1996}.
\end {proof}
\end {lemma}

\section{On stable CMC surfaces separating the compact part from infinity}\label{app:QT}

Here we explain how to modify the proof of Theorem \ref{thm:main} to obtain an alternative argument for the key technical step (due to J.~Qing and G.~Tian) in the proof of Theorem \ref{Qing-Tian:2007}, though under the \textsl{additional assumption} that the scalar curvature of $(M, g)$ is non-negative. Let $(M,g)$ be as in the statement of Theorem \ref{thm:main}.

\begin{proof}
Let $\{\Sigma_k\}_{k=1}^\infty$ be a sequence of stable constant mean curvature spheres $\Sigma_k\subset M$ each enclosing $B_1$ and such that $\area(\Sigma_k)\to\infty$ and $r_0(\Sigma_k)\to\infty$. Assume, for a contradiction, that $r_0(\Sigma_k)H(\Sigma_k) \to 0$. We follow the proof of Theorem \ref{thm:main} in Section \ref{sec:proof-thm-main} up until the flux integral \eqref{eq:outlying-flux-vanishes}, which needs to be replaced by 
\[
\int_{\Sigma_k} \frac { \bar g (X, \bar \nu)}{|x|^3} d \bar \mu = 4\pi
\]
since $\Sigma_k$ endloses $B_1$. Following along, instead of estimate \eqref{eq:BIG-inequality-with-Gamma} we obtain
\begin{align*}
& \left( \frac {2}{3}(1-\delta)+ 2 - \tau \Gamma \right)      \int_{\Sigma_k} |\mathring{\bar h}|^{2}_{\bar g} d\bar \mu +  4 \, m^2 \Big(1 - \frac2\tau \Big)  \int_{\Sigma_k} \frac{\bar g(X,\bar\nu)^{2}}{|x|^{6}} d\bar\mu - \frac{8\pi}{\lambda_k} \\ 
& \qquad \leq O(r_{0}(\Sigma_k)^{-3}) + H(\Sigma_k) \, O(r_{0}(\Sigma_k)^{-1}). 
\end{align*}
Using Lemma \ref{lem:r-2-lemm} as before, we arrive at 
\[
r_{0}(\Sigma_k)^{-2} \leq  O(r_{0}(\Sigma_k)^{-3}) + H(\Sigma_k)\, O(r_{0}(\Sigma_k)^{-1}) + \frac{8\pi}{\lambda_k} + O(1)  \int_{\Sigma_k} |\mathring{\bar h}|^{2}_{\bar g} d\bar \mu.
\]
The proof of Lemma \ref{lemm:prelim-main-thm} and Corollary \ref{cor:C1alpha} go through without change since $2/\lambda_k \sim H(\Sigma_k) = 2 / \rho_k$. We may thus assume that $\tilde \Sigma_k = \rho_k^{-1} \Sigma_k$ is a $C^{1, \alpha}$ graph of size $o(1)$ over a sphere $S_{1}(\xi)$  with $|\xi| =1$. 
We may argue as in \cite[(5.13)]{Huisken-Yau:1996} or \cite[Lemma 5.1]{Qing-Tian:2007} that
\begin{equation}\label{eq:HY-detect-dat-mass}
\int_{\tilde\Sigma_{k}} \frac{\bar g(\xi,\bar\nu)}{|x|} d\bar\mu + \int_{\tilde\Sigma_{k}} \frac{\bar g(X,\bar \nu)\bar g(\xi,\bar \nu)}{|x|^{3}} d\bar\mu = o(1).
\end{equation}
As in \cite[(5.13)]{Huisken-Yau:1996}, the first integral satisfies\footnote{Note we have chosen to use $\xi$ rather than $-\xi$ here, which is why the resulting integral is negative}
\[
\int_{\tilde\Sigma_{k}} \frac{\bar g(\xi,\bar\nu)}{|x|} d\bar\mu = \int_{S_{1}(\xi)} \frac{\bar g(\xi,\bar\nu)}{|x|} d\bar\mu + o(1) = - \frac{4\pi}{3} + o(1).
\]
To study the second integral in \eqref{eq:HY-detect-dat-mass}, we perturb $\xi$ slightly to $\xi_{k}\in\mathbb{R}^{3}$ so that $\tilde\Sigma_{k}$ and $S_{1}(\xi_{k})$ meet tangentially at the point $p_k \in \tilde\Sigma_{k}$ closest to the origin. (The transversality argument in the proof of Proposition \ref{prop:largestablecmcsphere} shows that these closest points are unique.) It follows that 
\[
|\bar \nu - (X-\xi_{k})| \leq o(1) |x-p_{k}|^{\alpha}.
\]

We estimate 
\begin{align*}
\int_{\tilde\Sigma_{k}} \frac{\bar g(X,\bar \nu)\bar g(\xi,\bar \nu)}{|x|^{3}} d\bar\mu & = \int_{\tilde \Sigma_{k}}\frac{\bar g(X,\bar \nu)\bar g(\xi,X-\xi_{k})}{|x|^{3}} d\bar\mu + \int_{\tilde\Sigma_{k}} \frac{\bar g(X,\bar \nu)\bar g(\xi,\bar \nu - (X-\xi_{k}))}{|x|^{3}} d\bar\mu\\
& =  \int_{\tilde \Sigma_{k}}\frac{\bar g(X,\bar \nu)\bar g(\xi,X)}{|x|^{3}} d\bar\mu -\bar g(\xi,\xi_{k}) \int_{\tilde \Sigma_{k}}\frac{\bar g(X,\bar \nu)}{|x|^{3}} d\bar\mu \\
& \qquad + \int_{\tilde\Sigma_{k}} \frac{\bar g(X,\bar \nu)\bar g(\xi,\bar \nu - (X-\xi_{k}))}{|x|^{3}} d\bar\mu\\
& =  \int_{S_{1}(\xi)}\frac{\bar g(X,\bar \nu)\bar g(\xi,X)}{|x|^{3}} d\bar\mu +o(1) +  (1+o(1)) \int_{\tilde \Sigma_{k}}\frac{\bar g(X,\bar \nu)}{|x|^{3}} d\bar\mu \\
& \qquad + \int_{\tilde\Sigma_{k}} \frac{\bar g(X,\bar \nu)\bar g(\xi,\bar \nu - (X-\xi_{k}))}{|x|^{3}} d\bar\mu\\
& = \frac{4\pi}{3} - 4\pi + o(1) +  \int_{\tilde\Sigma_{k}} \frac{\bar g(X,\bar \nu)\bar g(\xi,\bar \nu - (X-\xi_{k}))}{|x|^{3}} d\bar\mu. 
\end{align*}
In the second to last step, we use the uniform integrability of the first integrand. The last step follows from explicit computation of the first term and an application of the divergence theorem for the second term. Note that 
\[
\int_{\tilde\Sigma_{k}} \frac{\bar g(X,\bar \nu)\bar g(\xi,\bar \nu - (X-\xi_{k}))}{|x|^{3}} d\bar\mu = o(1) \int_{\tilde\Sigma_{k}} \frac{|x-p_{k}|^{\alpha}}{|x|^{2}} d\bar\mu= o(1).
\]
Putting everything together, we have 
\[
\int_{\tilde\Sigma_{k}} \frac{\bar g(\xi,\bar\nu)}{|x|} d\bar\mu + \int_{\tilde\Sigma_{k}} \frac{\bar g(X,\bar \nu)\bar g(\xi,\bar \nu)}{|x|^{3}} d\bar\mu = -4\pi + o(1).
\]
This estimate contradicts \eqref{eq:HY-detect-dat-mass}. The proof is now finished exactly as in \cite[Section 5]{Huisken-Yau:1996}.
\end{proof}


\section{Estimates for the Willmore deficit} \label{sec:Willmore-def} 

An estimate of the form \eqref{CYHI} has been proven in \cite{Chodosh:large-iso, CESY} for isoperimetric regions. As noted in Appendix C of \cite{Chodosh:thesis}, the proofs carry over to the case where $\Sigma$ is not necessarily outward minimizing. We adapt these ideas below. 

\begin{proposition}\label{prop:IMCF-CY}
Let $(M, g)$ be a connected, complete Riemannian $3$-manifold with non-negative scalar curvature and which is $C^{2}$-asymptotically flat of rate $q>1/2$. Consider a sequence $\{\Sigma_k\}_{k=1}^\infty$ of stable constant mean curvature spheres $\Sigma_k \subset M$ such that $r_0(\Sigma_k)\to\infty$ and $\area(\Sigma_k)\to\infty$. Then,
\begin{align} \label{CYHI}
\area (\Sigma_k)^{\frac{1}{2}} \int_{\Sigma_k} |\mathring h|^{2} d\mu = O(1)
\end{align}
as $k \to \infty$.

\begin{proof}
Since $r_0(\Sigma_k) \to\infty$, we may assume that $(M,g)$ has horizon boundary; cf.~\cite[Lemma 4.1]{Huisken-Ilmanen:2001}. Let $\Omega_k \subset M$ be the unique compact region $\Omega_k \subset M$ with $\Sigma_k = \partial \Omega_k$. Let $\Omega'_k \subset M$ be the strictly minimizing hull of $\Omega_k$ in $(M, g)$; cf.~\cite[p.~371]{Huisken-Ilmanen:2001}. Recall from \cite[Theorem 1.3]{Huisken-Ilmanen:2001} that the boundary $\Sigma'_k$ of $\Omega'_k$ is $C^{1,1}$ and smooth away from $\Sigma_k$. By \cite[(1.15)]{Huisken-Ilmanen:2001}, the weak mean curvature of $\Sigma'_k$ satisfies $H_{\Sigma'_k} = 0$ on $\Sigma'_k\setminus\Sigma_k$ and $H_{\Sigma'_k} = H_{\Sigma_k}$ for $\mathscr{H}^{2}$-a.e.~point of $\Sigma'_k\cap \Sigma_k$. In particular,
\[
\int_{\Sigma'_k} H_{\Sigma'_k}^{2} d\mu \leq \int_{\Sigma_k} H_{\Sigma_k}^{2}d\mu.
\]
There is a weak solution in the sense of \cite[p.~365]{Huisken-Ilmanen:2001} to inverse mean curvature flow starting at $\Sigma'_k$  by   \cite[Lemma 5.6]{Huisken-Ilmanen:2001}. The monotonicity of the Hawking mass along the flow \cite[(5.24)]{Huisken-Ilmanen:2001} in combination with \cite[Lemma 7.3]{Huisken-Ilmanen:2001} gives\footnote{The argument in \cite{Huisken-Ilmanen:2001} requires that $(M, g)$ be asymptotically flat of rate $q = 1$. Note that this case suffices for the application in the proof of Theorem \ref{thm:main}. To cover the full range $q > 1/2$, we can argue exactly as in Appendix H of \cite{CESY}.} 
\[
 \area(\Sigma'_k)^{\frac 12} \left(16\pi - \int_{\Sigma'_k} H_{\Sigma'_k}^{2}d\mu\right) \leq (16\pi)^{ \frac 32} m_{ADM}
\]
so that
\begin{align} \label{aux1:HICY}
 \area(\Sigma'_k)^{\frac 12} \Big( 16\pi - \int_{\Sigma_k} H_{\Sigma_k}^{2}d\mu \Big)  \leq (16\pi)^{ \frac 32} m_{ADM}.
\end{align}
The Christodoulou--Yau estimate \eqref{eq:CY-intro} gives
\begin{align} \label{aux2:HICY}
\frac{2}{3} \int_{\Sigma_k} |\mathring h|^{2} d\mu \leq 16\pi - \int_{\Sigma_k} H_{\Sigma_k}^{2}d\mu. 
\end{align}
Finally, we note that
\begin{align} \label{aux3:HICY}
\area(\Sigma_k) = (1 + o(1)) \area(\Sigma'_k).
\end{align}
Indeed, this follows from the rescaling arguments in Proposition \ref{prop:area-rescaling} and Proposition \ref{prop:area-rescaling-trivial} and a coarse, Euclidean area comparison. The asserted estimate now follows from combining \eqref{aux1:HICY}, \eqref{aux2:HICY}, and \eqref{aux3:HICY}.
\end{proof}
\end{proposition}


\section{A remark on the Minkowski inequality} \label{sec:Minkowski}

The contribution here is an estimate for the remainder in the second order Taylor expansion of the Minkowski quantity 
\[
\int_\Sigma H - \sqrt {16 \pi \area (\Sigma)}
\]
at the unit sphere of Euclidean space. The idea of computing and using the second variation of this quantity appears in the thesis of D.~Perez \cite{Perez:thesis}.

All geometric quantities in this section are computed with respect to the Euclidean metric.  We abbreviate  
\[
S = \{x \in \R^3 : |x| = 1\}.
\]

Let $f  \in C^2(S)$ with $\Vert f\Vert_{C^{1}(S)}$ small, say $\Vert f\Vert_{C^{1}(S)} < 1/2$. Let 
\[
\Sigma = \{ (1 + f(\theta)) \theta : \theta \in S\}. 
\]

\begin{lemma}
We have  
\begin {align} \label{eqn:TaylorMinkowski}
\int_{\Sigma} H - \sqrt{16\pi \area(\Sigma)} =&  
\frac{1}{4\pi} \Big( \int_{S} f \Big)^{2} - \int_{S} f^{2} + \frac{1}{2}  \int_S |\nabla f|^2 \\ 
& \qquad + O (1) \int_{S} (|f|^3 + |\nabla f|^3 +  |\nabla d f| |\nabla f|^2  + |\nabla d f| f^2) \nonumber
\end{align}
as $f\to 0$ in $C^{1}(S)$.

\begin{proof}
Let  
\[
\mathscr{M} (t) = \int_{\Sigma_{t}} H_{\Sigma_t}  - \sqrt{16\pi \area(\Sigma_{t})}
\]
for $t \in [0, 1]$ where $\Sigma_t \subset \R^3$ is the surface 
\[
\Sigma_{t} = \{ (1 + t \, f(\theta)) \theta : \theta \in S\}.
\]
Taylor's theorem gives 
\[
\mathscr{M}(1) = \mathscr{M}(0) + \mathscr{M}'(0) + \frac 12 \mathscr{M}''(0) + \frac 16 \mathscr{M}'''(t)
\]
for some $t \in (0, 1)$. The computation in the proof of Proposition 4.1 in \cite{Perez:thesis} shows that  
\begin{align*}
\mathscr{M}(0) &= 0\\
\mathscr{M}'(0) &= 0\\
 \mathscr{M}''(0) &=  \frac{1}{2\pi} \left(\int_{S} f\right)^{2} - 2 \int_{S} f^{2} + \int_{S} |\nabla f|^{2}.
\end{align*}
To see that $\mathscr{M}'''(t)$ has the asserted form, we recall the quasi-linear structure of mean curvature below. 
\end{proof}
\end{lemma}

We briefly recall the quasi-linear structure of mean curvature that is used in the proof of \eqref{eqn:TaylorMinkowski}. Let $U \subset \R^m$ be open. The mean curvature vector field of an immersion $\varphi: U \to \R^n$ is given by  
\[
g^{ij} (\partial_i \partial_j \varphi - \Gamma_{ij}^k \partial_k \varphi).
\]
Here, 
\[
g_{ij} = (\partial_i \varphi) \cdot (\partial_j \varphi)
\]
are the components of the first fundamental form, $g^{ij}$ are the components of its inverse, and 
\[
\Gamma_{ij}^k = \frac{1}{2} g^{k \ell} \big ( \partial_i g_{\ell j } + \partial_j g_{ i \ell} - \partial_\ell g_{ij} \big)
\]
are the Christoffel symbols. This expression is \textsl{linear} in the second order partial derivatives of the immersion. \\

Next, note that there is a universal constant $\delta \in (0, 1/2)$ -- independent of a particular choice of $f$ -- with the following property. For every $v \in  \overline {B_\delta (0)}$, the translate $\Sigma + v$ is a $C^2$ graph over $S$. \\

\begin{lemma} Assume that $\Sigma \subset \R^3$ is the graph of a function $f \in C^1(S)$ with small norm. There is $v \in \R^3$ small such that the translate $\Sigma' = \Sigma + v$ is the graph of a function $f' \in C^1(S)$ with small norm so that 
\begin{align} \label{eqn:moments}
\int_S x^1 \, f'  = \int_S x^2 \, f'  = \int_S x^3 \, f'  =0.
\end{align}
We may homothetically rescale $\Sigma'$ slightly to the graph of a function $f'' \in C^1(S)$  where
\begin{align} \label{eqn:momentzero}
\int_S f'' = 0 \qquad \text{ and } \qquad \int_S x^1 \, f''  = \int_S x^2 \, f''  = \int_S x^3 \, f''  =0.
\end{align}

\begin{proof}
Let $f_v \in C^1(S)$ be such that  
\[
\{ v + (1 + f(\theta)) \theta : \theta \in S \} = \Sigma + v = \{  (1 + f _v (\theta)) \theta : \theta \in S\}.
\]
For $f=0$, note that $f_{v} (\theta) =  \theta \cdot v+ O(|v|^{2})$. In general, we find 
\[
f_{v} (\theta) =  \theta \cdot v + O(|v|^{2}) + O(\Vert f \Vert_{C^{1}}).
\]
Let $\delta \in (0, 1/2)$ be small. Consider the map
\[
\Xi : \{ v \in \R^3 : |v| \leq \delta\} \to \mathbb{R}^{3} \qquad \text{ given by } \qquad v \mapsto v - \frac{3}{4 \pi}\int_S (x^1, x^2, x^3) f_v.
\]
It follows that 
\[
\Xi(v) = O(|v|^{2}) + O(\Vert f\Vert_{C^{1}}). 
\]
If $\Vert f\Vert_{C^{1}}$ is sufficiently small, then this map has a fixed point $v$ by Brouwer's fixed point theorem. 

Finally, if $\Sigma$ is the graph of $f' \in C^1(S)$ satisfying \eqref{eqn:moments}, then (the homothetic rescaling) $\lambda \, \Sigma$ is the graph of $\lambda -1 + \lambda f'$. For any $\lambda>0$, this still satisfies \eqref{eqn:moments}. Choosing
\[
\lambda = \left( 1 + \frac{1}{4\pi} \int_S f'\right)^{-1} = 1 + O(\Vert f\Vert_{C^0}),
\]
we find $f''$ satisfying \eqref{eqn:momentzero}. 
\end{proof}
 
\end{lemma}

Clearly,
\begin{align}
\int_{S} ( |\nabla d f| |\nabla f|^2  + |\nabla d f| f^2 + |\nabla f|^3 + |f|^3) = o(1)  \int_S (f^2 + |\nabla f|^2 +  |\nabla d f|^2)
\end{align}
as $f\to 0$ in $C^{1}(S)$. Moreover, the Bochner formula on the sphere gives  
\begin{align} \label{eqn:Bochner}
\int_S  |\nabla d f|^2 = 2 \int_S |\nabla d f - (\Delta f /2) g|^2 + 2 \int_S |\nabla f|^2.
\end{align}

\begin{lemma} \label{lem:tracefreehessian}
We have 
\[
 \int_S |\nabla d f - (\Delta f / 2) g|^2  \leq O(1) \int_\Sigma |\mathring h|^2 + o (1) \int_S ( f^2 + |\nabla f|^2)
\]
as $f \to 0$ in $C^1(S)$. 
\end{lemma}
\begin{proof}
Lemma 2.3 in \cite{PacardXu} shows that
\begin{align*}
\mathring h &= \nabla df - (\Delta f/2) g \\
& \qquad + O(|f||\nabla d f| + |\nabla f||\nabla d f|) + O(f^2 + |\nabla f|^2).
\end{align*}
We now use \eqref{eqn:Bochner} to absorb the error terms containing $\nabla d f$. 
\end{proof}

From now on, we assume that the moment conditions  \eqref{eqn:moments} and \eqref{eqn:momentzero} hold. Then 
\[
\int_S |\nabla f|^2 \geq 6 \int_S f^2
\]
from which the estimate 
\[
\frac{1}{4\pi} \Big( \int_{S} f \Big)^{2} -  \int_{S} f^{2} + \frac{1}{2} \int_S |\nabla f|^2 \geq \frac{1}{6} \int_S (f^2 +  |\nabla f|^2 )
\]
follows. Putting this together, we obtain 
\[
\sqrt{16\pi \area(\Sigma)} \leq \int_{\Sigma} H  + o(1) \int_{S} |\nabla df - (\Delta f/2)g|^{2} - \frac 1 7 \int_{S} (f^{2} + |\nabla f|^{2})
\]
as $f \to 0$ in $C^1(S)$.  In combination with Lemma \ref{lem:tracefreehessian}, we obtain our final result in this section.

\begin{proposition}\label{prop:improvedminkowsi}
We have that 
\begin {align} \label{improvedminkowsi}
\sqrt{16\pi \area(\Sigma)} \leq   \int_{\Sigma} H +o (1) \int_\Sigma |\mathring h|^2
\end{align}
as $\Sigma$ converges to $S$ in $C^1$. 
\end{proposition}

\begin {remark}
For $\Sigma$ close enough to $S$ in $C^2$,  by the classical Minkowski inequality for \textsl{convex surfaces}, 
\[
\sqrt{16\pi \area(\Sigma)} \leq   \int_{\Sigma} H.
\]
The Minkowski inequality has been generalized to mean-convex star-shaped surfaces \cite{GuanLi,GuanMaTrudingerZhu,BrendleHungWang}, see also Theorem 3.3 in  \cite{Perez:thesis}, as well as to outward minimizing surfaces \cite{Huisken:minkowski}.
\end {remark}


\bibliography{bib} 
\bibliographystyle{amsplain}
\end{document}